\documentclass[reqno,12pt]{amsart}
\usepackage{stmaryrd}
\usepackage{amssymb}
\usepackage{amsfonts}
\usepackage{amsmath, amsthm, amscd, amsfonts}
\usepackage{cite}
\usepackage{geometry}
\allowdisplaybreaks
\newtheorem{theorem}{Theorem}
\newtheorem{lemma}{Lemma}

\theoremstyle{definition}

\newtheorem{example}[theorem]{Example}

\newtheorem{problem}{Problem}
\theoremstyle{remark}
\newtheorem{remark}{Remark}
\newtheorem{conclusion remark}[theorem]{Conclusion Remark}

\begin{document}
\setcounter{page}{1}

\title[ complex symmetric weighted composition operators]
      {A  note on the complex symmetric weighted composition operators over Hardy space}

\author[ C. Jiang, SA. Han,  ZH. Zhou]{Cao Jiang, Shi-An Han,  Ze-Hua Zhou$^*$ }


\address{\newline Cao Jiang\newline School of Mathematics, Tianjin University,
 Tianjin 300354, P.R. China.}
\email{jiangcc96@163.com}

\address{\newline  Shi-An Han\newline College of Science, Civil Aviation University of China,
 Tianjin 300300, P.R. China.}
\email{hsatju@163.com}

\address{\newline  Ze-Hua Zhou\newline School of Mathematics, Tianjin University,
Tianjin 300354, P.R. China¡£} \email{zehuazhoumath@aliyun.com;
zhzhou@tju.edu.cn}

\date{}
\thanks{\noindent $^*$Corresponding author.\\
This work was supported in part by the National Natural Science Foundation of
China (Grant Nos. 11771323; 11371276).}
\begin{abstract}
This paper provides a class of complex symmetric weighted composition operators on $H^2(\mathbb{D})$ to includes  the unitary subclass, the Hermitian subclass and the normal subclass  obtained  by Bourdon and Noor.  A characterization of algebraic weighted composition operator with degree no more than two  is provided to illustrate that the weight function of a complex symmetric weighted composition operator is not necessarily  linear fractional.
\end{abstract}

\maketitle
\section{Preliminaries}
\subsection{Complex Symmetricity}
Let $\mathcal{H}$ be a  complex Hilbert space  and $\mathcal{L}(\mathcal{H})$ be the collection of all  continuous linear operators on  $\mathcal{H}$.
A map $\mathcal{C}: \mathcal{H}\rightarrow \mathcal{H}$ is called a \textit{conjugation} over $\mathcal{H}$ if it is

$\bullet$ anti-linear: $\mathcal{C}(ax+by)=\overline{a}\mathcal{C}(x)+\overline{b}\mathcal{C}(y)$, $ x, y\in\mathcal{H}$,
                   $ a,b\in \mathbb{C}$;

$\bullet$ isometric: $\|\mathcal{C}x\|=\|x\|$, $ x\in\mathcal{H}$;

$\bullet$ involutive: $\mathcal{C}^2=I$.

 \noindent An operator $T\in \mathcal{L}(\mathcal{H})$  is   called \textit{complex symmetric}  if
$$\mathcal{C}T=T^*\mathcal{C},$$
for some conjugation $\mathcal{C}$ and in this case we say $T$ is \textit{complex symmetric with conjugation} $\mathcal{C}$.

The general study of complex symmetric operators was started by Garcia, Putinar and Wogen in \cite{GP1,GP2,GW1,GW2}.
The class of  complex symmetric operators turns to be quite diverse, see \cite{GP1,GP2,GW2}. In this paper, we focus  on two classes: the normal operator and  the algebraic operator with degree no more than two.
We know from the  Spectral Theorem  that  a normal operator is unitarily equivalent to some multiplier
$M_\phi:L^2(X,v)\rightarrow L^2(X,v). $
Take the usual conjugation $Cf(z)=\overline{f(z)}$, then  it is easy to check that $M_\phi$ is complex symmetric.  An operator $T\in\mathcal{L}(\mathcal{H})$ is said to be  \textit{algebraic} if it is annihilated by some nonzero polynomial $p$ and the minimal degree of $p$  is called the \textit{degree} of $T$.  Garcia and Wogen \cite{GW2} proved that an  algebraic operator with degree no more than two is  complex symmetric.

\subsection{Hardy Space}
The classical Hardy-Hilbert  space is defined by
$$H^2(\mathbb{D})=\left\{f\in H(\mathbb{D});\,\, \|f\|^2
  =\sup\limits_{0<r<1} \frac{1}{2\pi}\int_0^{2\pi}|f(re^{i\theta})|^2d\theta<\infty \right\}.  $$ $H^2(\mathbb{D})$ is a
reproducing kernel Hilbert space (RKHS), i.e., for any point $w\in\mathbb{D}$ there exists a unique function
$K_w\in H^2(\mathbb{D})$ such that $$f(w)=\langle f, K_w\rangle,$$ where $K_w=1/(1-\overline{w}z)$ is called the \textit{reproducing kernel }at $w$.

\subsection{Weighted Composition Operator}
Denote by $H(\mathbb{D})$ the set of all holomorphic functions over the unit disk and by $S(\mathbb{D})$ the  set of all holomorphic selfmaps  of the unit disk.
For any $\psi\in H(\mathbb{D})$ and $\varphi\in S(\mathbb{D})$,
the associated \textit{weighted composition operator} is defined by
$$W_{\psi,\varphi}(f)=\psi\cdot(f\circ\varphi).$$
When $u\equiv1$, we write $W_{1,\varphi}=C_\varphi$ which is the  \textit{composition operator}. The study of composition operator and weighted composition
 operator over  various holomorphic function spaces  has undergone a rapid development over the past four decades. For general background, see \cite{CM} and \cite{S}.

Composition operators on $H^2(\mathbb{D})$ are relatively well understood. For example, every normal composition operator $C_\varphi$ on  $H^2(\mathbb{D})$ must be  induced by a dilation $\varphi(z)=az, |a|\leq 1.$ However, the case for weighted composition operator  is more complicated.   There exists many nontrivial normal even Hermitian weighted composition operators on $H^2(\mathbb{D})$, see \cite{BNa,CK}. We will talk  more about it later.

The study of  complex symmetric weighted composition operators  on $H^2(\mathbb{D})$
 was initiated  independently by  Garcia and Hammond  in \cite{GH} and Jung \textit{et al.} in \cite{JKKL}.
Their idea is to consider the following  conjugation
$$ \mathcal{J} (\sum\limits_{n=0}^\infty a_nz^n)= \sum\limits_{n=0}^\infty \overline{a}_nz^n$$
and to characterize when $W_{\psi,\varphi}$ is complex symmetric under  $\mathcal{J}$. Their main result is the  following theorem:

\textbf{Theorem A }
\emph{Let $\varphi\in S(\mathbb{D})$ and $\psi\in H^\infty(\mathbb{D})$. Then  $W_{\psi.\varphi}$ is complex symmetric under the conjugation $\mathcal{J}$ iff
$$\psi(z)=\frac{b}{1-a_0z}\,\, and\,\, \varphi(z)=a_0+\frac{a_1z}{1-a_0z}.$$}

\noindent By making unitary transformations, Jung et al. \cite{JKKL}  provided some more examples. For more work on complex symmetric composition operators, see \cite{N2,NST,BNo,GZ2}.

In this paper, we will  make a slight  generalization of Jung's  result. As a corollary, we will give conjugations for the unitary weighted composition operator and the Hermitian weighted composition operator. To our surprise, we find that the normal subclass of our complex symmetric $W_{\psi,\varphi}$ coincide with those given by Bourdon and Noor in their early paper \cite{BNa}. We do not know whether this is  an accident or a hint that there are no more normal $W_{\psi,\varphi}$. It is still an open question to find a complete characterization of the normal $W_{\psi,\varphi}$ on $H^2(\mathbb{D})$.   Besides, we will also try to characterize all algebraic $W_{\psi,\varphi}$ of degree no more than two, since by doing so we will show that the weight function of a complex symmetric weighted composition operator is not necessarily linear fractional.

\section{Main Results}
\subsection{The Conjugation}
The main difficulty in constructing complex symmetric weighted composition operators is to find conjugation which is feasible to do calculation.  If  $\mathcal{C}_1$ and $\mathcal{C}_2$ are conjugations, then  $\mathcal{C}_1\mathcal{C}_2$ is a unitary operator. So  any two conjugations is related by  some unitary operator, i.e. $\mathcal{C}_2= \mathcal{C}_1(\mathcal{C}_1\mathcal{C}_2).$
Following this idea, we would get new conjugation if we composite an  unitary operator to the conjugation $\mathcal{J}$.  The choice of the unitary operator is the
unitary weighted composition operator in the following lemma.

\begin{lemma}\cite[Theorem 6]{BNa}\label{normal}
A weighted composition operator $W_{\psi,\varphi}$ is unitary on $H^2(\mathbb{D})$ if and only if $\varphi$ is an automorphism of the unit disk and $$\psi=\mu\frac{K_p}{\|K_p\|}=\mu\frac{\sqrt{1-|p|^2}}{1-\overline{p}z} $$ where $p=\varphi^{-1}(0)$ and $|\mu|=1$.
\end{lemma}

\begin{lemma}\label{conjugation}

 Suppose that  $$\sigma(z)=\lambda\frac{p-z}{1-\overline{p}z}\ \ and\ \
 k_p(z)=\frac{\sqrt{1-|p|^2}}{1-\overline{p}z},$$
 where $p\in\mathbb{D}$ and $|\lambda|=1$.
Then  $\mathcal{J}W_{k_p,\sigma}$ defines a conjugation on $H^2(\mathbb{D})$ if and only if  $\overline{p}=\lambda p$.
 \end{lemma}
\begin{proof}
By Lemma~\ref{normal},  $\mathcal{J}W_{k_p,\sigma}$ is anti-linear and isometric, so it remains to analyze when $\mathcal{J}W_{k_p,\sigma}$ is an involution.

For any $f\in H^2(\mathbb{D})$, we have
$$\mathcal{J}W_{k_p,\sigma}f(z)=\mathcal{J}(k_p(z)f(\sigma(z)))=\overline{k_p(\overline{z})}\overline{f(\sigma(\overline{z}))},$$
and then
\begin{eqnarray*}
  (\mathcal{J}W_{k_p,\sigma})^2f(z) &=& \mathcal{J}W_{k_p,\sigma}\overline{k_p(\overline{z})}\overline{f(\sigma(\overline{z}))} \\
   &=& \overline{k_p(\overline{z})}k_p\left(\overline{\sigma(\overline{z})}\right)f(\sigma(\overline{\sigma(\overline{z})})).
\end{eqnarray*}

If  $\mathcal{J}W_{k_p,\sigma}$ is involutive, then by taking $f\equiv1$ we have
\begin{eqnarray*}
   1 &=& \overline{k_p(\overline{z})}k_p\left(\overline{\sigma(\overline{z})}\right) \\
   &=& \frac{\sqrt{1-|p^2|}}{1-pz}\frac{\sqrt{1-|p^2|}}{1-\overline{p}\overline{\lambda}\frac{\overline{p}-z}{1-pz}}\\
   &=&\frac{1-|p^2|}{1-\overline{\lambda p^2}-(p-\overline{p\lambda})z},
\end{eqnarray*}
which implies  $\lambda p=\overline{p}$.

Conversely, if we assuming $\lambda p=\overline{p}$, then
$$\overline{k_p(\overline{z})}k_p(\overline{\sigma(\overline{z})})=1,$$
and
\begin{eqnarray*}
  \sigma(\overline{\sigma(\overline{z})}) &=& \lambda\frac{p-\overline{\lambda}\frac{\overline{p}-z}{1-pz}}
{1-\overline{p}\overline{\lambda}\frac{\overline{p}-z}{1-pz}}\\
  &=& \frac{\lambda p-\lambda p^2z-\overline{p}+z}{1-pz-\overline{\lambda}\overline{p}^2+\overline{\lambda}\overline{p}z} \\
  &=&   \frac{z-|p|^2z}{1-|p|^2},\,\,\,\,(\lambda p=\overline{p})\\
   &=& z.
\end{eqnarray*}

So  $(\mathcal{J}W_{k_p,\sigma})^2=I$, that is, $\mathcal{J}W_{k_p,\sigma}$ is a conjugation.
\end{proof}
\begin{remark}
Throughout the paper,  the symbol $\mathcal{J}W_{k_p,\sigma}$ is always referred to  the conjugation constructed above, that is,
$\lambda=\frac{\overline{p}}{p}$ if $p\neq 0$ and $\lambda$ is any unimodular constant if $p=0$.

\end{remark}

\subsection{Complex Symmetry}

Now we will investigate when a  weighted composition operator  $W_{\psi,\varphi}$ is complex symmetric under the conjugation $\mathcal{J}W_{k_p,\sigma}$.
Having the result of Jung \textit{et al.} at hand, this is not a difficult question to answer. We list the result in the following theorem.

\begin{theorem}\label{cwc}
Let $\varphi\in S(\mathbb{D})$, $\psi\in H^\infty(\mathbb{D})$. Then  $W_{\psi,\varphi}$ is complex symmetric under
 the conjugation $\mathcal{J}W_{k_p,\sigma}$ where  $$\sigma(z)=\lambda\frac{p-z}{1-\overline{p}z}\ \ and\ \
 k_p(z)=\frac{\sqrt{1-|p|^2}}{1-\overline{p}z}$$
with $p\in\mathbb{D}$ and $|\lambda|=1$ if and only if
$$\psi(z)=\frac{c}{1-a_0p-(p-\overline{\lambda}a_0)z} \ \ and \ \
\varphi(z)=a_0+\frac{a_1(p-\overline{\lambda}z)}{1-a_0p-(p-\overline{\lambda}a_0)z}.$$
\end{theorem}
\begin{proof}
By definition, $W_{\psi,\varphi}$ is complex symmetric  under the conjugation $\mathcal{J}W_{k_p,\sigma}$ if and only if
\begin{equation}\label{J}
   \mathcal{J}W_{k_p,\sigma}W_{\psi,\varphi}
   =W_{\psi,\varphi}^*\mathcal{J}W_{k_p,\sigma}.
\end{equation}
Note that $\mathcal{J}W_{k_p,\sigma}$ is a conjugation, so  $$ W^*_{k_p,\sigma}=\mathcal{J}W_{k_p,\sigma}\mathcal{J},$$
and then  (\ref{J}) is equivalent to $$W_{k_p\cdot\psi\circ\sigma,\varphi\circ\sigma}\mathcal{J}=\mathcal{J}W_{k_p\cdot\psi\circ\sigma,\varphi\circ\sigma}^*.$$

It follows  from Theorem A that $W_{k_p\cdot\psi\circ\sigma,\varphi\circ\sigma}$
is complex symmetric under the conjugation  $\mathcal{J}$ if and only if
$$k_p(z)\cdot\psi(\sigma(z))=\frac{b}{(1-a_0z)} \ \ and \ \ \varphi(\sigma(z))=a_0+\frac{a_1z}{1-a_0z}.$$
Consequently, we have
\begin{eqnarray*}
  \psi(z) &=& \frac{k_p\cdot\psi\circ\sigma}{k_p}\circ\sigma^{-1}(z)\\&=&\frac{b\sqrt{1-|p|^2}}{1-a_0p-(p-\overline{\lambda}a_0)z} \\
   &=&\frac{c}{1-a_0p-(p-\overline{\lambda}a_0)z}
\end{eqnarray*}
and
\begin{eqnarray*}
  \varphi(z)&=&\varphi\circ\sigma\circ\sigma^{-1}(z)\\
  &=&a_0+\frac{a_1(p-\overline{\lambda}z)}{1-a_0p-(p-\overline{\lambda}a_0)z}.
\end{eqnarray*}
This completes the proof.
\end{proof}

\subsection{Relation with the Normal Class} Generally speaking, the class of complex symmetric operators includes the normal operators, particularly the unitary operators and the Hermitian operators.
In this part, we will  apply Theorem~\ref{cwc} to  give a conjugation with which a unitary (Hermitian, normal resp.) weighted composition operator is complex symmetric.

We first consider the unitary class and Hermitian class. For the unitary weighted composition operator  listed in Lemma \ref{normal}, the conjugation is given by the following theorem.
\begin{theorem}
For the unitary weighted composition operator $W_{\psi,\varphi}$ where
 $$\varphi(z)=\mu_1\frac{q-z}{1-\overline{q}z}\ \ and\ \
 \psi(z)=\mu_2\frac{\sqrt{1-|q|^2}}{1-\overline{q}z}$$
 with $q\in\mathbb{D}$ and $|\mu_1|=|\mu_2|=1$, it is complex symmetric with conjugation $\mathcal{J}W_{k_{\overline{q}},\sigma}$. \end{theorem}
\begin{proof}
Let  $a_0=0$, $p=\overline{q}$, $a_1=\lambda\mu_1$ and $c=\mu_2\sqrt{1-|q|^2}$ in Theorem \ref{cwc}, the result then follows. \end{proof}

The class of Hermitian weighted composition operators on $H^2(\mathbb{D})$ is completely characterized by  Cowen and Ko \cite{CK}. For this, we have:
\begin{theorem}
For the Hermitian weighted composition operator $W_{\psi,\varphi}$ where
 $$\varphi(z)=b_0+\frac{b_1z}{1-\overline{b}_0z}\ \ and\ \
 \psi(z)=\frac{b_2}{1-\overline{b}_0z}$$
 with $b_0\in\mathbb{D}$ and $b_1,b_2\in\mathbb{R}$, it is complex symmetric with conjugation $\mathcal{J}C_{\lambda z}$ where $\lambda \overline{b_0}+b_0=0$.
\end{theorem}
 \begin{proof}
Let $a_0=b_0$, $c=b_2$, $p=0$ and $\lambda, a_1$ be such that $\lambda \overline{b_0}=-b_0,\  a_1\overline{\lambda}=-b_1$ in Theorem \ref{cwc}, then the result follows. \end{proof}

The case of normal class is a little   complicated.  Bourdon and  Narayan \cite{BNa} studied the normality of  weighted composition operators on $H^2(\mathbb{D})$.
Their main results are the following two theorems.

\vspace{2pt}

\textbf{Theorem B} \cite[Theorem 10]{BNa}  \emph{Suppose that $\varphi\in S(\mathbb{D})$ has a fixed point $p\in\mathbb{D}$. Then $W_{\psi,\varphi}$
acting on $H^2(\mathbb{D})$ is normal if and only if
$$\psi=\gamma\frac{K_p}{K_p\circ\varphi}\ \ and \ \ \varphi=\alpha_p\circ(\delta\alpha_p),$$
where $\alpha_p(z)=\frac{p-z}{1-\overline{p}z}$, $\gamma$, $\delta$ are constants with $\delta$ satisfying $|\delta|\leq1$.}

\textbf{Theorem C} \cite[Proposition 12]{BNa}  \emph{
Suppose that
$$\varphi(z)=\frac{az+b}{cz+d},\ \  ad-bc\neq 0$$
is a linear fractional selfmap of the unit disk and $\psi=K_{\varphi^*(0)}$,
where $\varphi^*(z)=\frac{\overline{a}z-\overline{c}}{-\overline{b}z+\overline{d}}$. Then $W_{\psi,\varphi}$ acting on $H^2(\mathbb{D})$
is normal if and only if
\begin{equation}\label{14}
  \frac{|d|^2}{|d|^2-|b|^2-(\overline{b}a-\overline{d}c)z}C_{\varphi^*\circ\varphi}
  =\frac{|d|^2}{|d|^2-|c|^2-(\overline{b}d-c\overline{a})z}C_{\varphi\circ\varphi^*}.
\end{equation}}

The above two theorems list all the normal weighted composition operators known up to now: Theorem B gives a complete characterization in the case when $\varphi$ has an interior fixed point; Theorem C gives some partial results in the case when $\varphi$ is a linear fractional  transformation.   For the remainder of this part, we will give a conjugation with which the normal $W_{\psi,\varphi}$ of the above theorems is complex symmetric.

For  Theorem \textmd{B}, checking its proof, one can find that $W_{\psi,\varphi}$ is actually unitarily equivalent to the composition operator $C_{\delta z}$ via the formula
\begin{eqnarray*}
  W_{\psi,\varphi} &=& W_{k_p,\alpha_p}\circ C_{\delta z}\circ W_{k_p,\alpha_p}^{-1} \\
  &=&W_{k_p,\alpha_p}\circ C_{\delta z}\circ W_{k_p^{-1}\circ\alpha_p,\alpha_P}.
\end{eqnarray*}
Hence it is easy to give a conjugation for  $W_{\psi,\varphi}$ in this case.

\begin{theorem}\label{B} For a normal weighted composition operator $W_{\psi,\varphi}$ where  $\varphi$  has an interior fixed point $p$, i.e.,
$$\psi=\gamma\frac{K_p}{K_p\circ\varphi}\ \ and \ \ \varphi=\alpha_p\circ(\delta\alpha_p),$$
where $\alpha_p(z)=\frac{p-z}{1-\overline{p}z}$, $\gamma$, $\delta$ are constants with $\delta$ satisfying $|\delta|\leq1$, it is
complex symmetric with conjugation
$\mathcal{J} W_{k_q,\sigma}$ where $q=\frac{p-\overline{p}}{p^2-1}$.
\end{theorem}
\begin{proof}
It is obvious that  $C_{\delta z}$ is complex symmetric with classical conjugation $\mathcal{J}.$  By the above argument, $W_{\psi,\varphi}$
is complex symmetric with conjugation $W_{k_p,\alpha_p}\circ \mathcal{J}\circ W_{k_p^{-1}\circ\alpha_p,\alpha_P}$. Elementary calculation will show that
$$ W_{k_p,\alpha_p}\circ \mathcal{J}\circ W_{k_p^{-1}\circ\alpha_p,\alpha_P}= \mu\mathcal{J} W_{k_q,\sigma},$$
where $q=\frac{p-\overline{p}}{p^2-1}$ and $\mu=\frac{|1-p|}{1-p}$. \end{proof}

For Theorem C, we will only consider  the case when $\varphi$ admits a boundary fixed point since the case when $\varphi$ admits an interior fixed point
is already answered in Theorem \ref{B}. Another reason is the following.

\begin{lemma}\label{bc}Suppose that
  $$\varphi(z)=\frac{az+b}{cz+d},\ \  ad-bc\neq 0$$
  is a linear fractional selfmap of the unit disk, which admits a boundary fixed point $\eta\in\mathbb{T}$,   and $\psi=K_{\varphi^*(0)}$                                                               where $\varphi^*(z)=\frac{\overline{a}z-\overline{c}}{-\overline{b}z+\overline{d}}$. Then $W_{\psi,\varphi}$ acting on $H^2(\mathbb{D})$
  is normal if and only if $|b|=|c|$.
\end{lemma}
\begin{proof}
We need to show that Eq.(\ref{14}) is equivalent to $|b|=|c|$.

If (\ref{14}) holds, then by taking test function $f\equiv 1$,  we get
$$
  \frac{|d|^2}{|d|^2-|b|^2-(\overline{b}a-\overline{d}c)z}
  =\frac{|d|^2}{|d|^2-|c|^2-(\overline{b}d-c\overline{a})z}.  $$
Since $d\neq 0$, we have $|b|=|c|$.

If $|b|=|c|$, since $$\varphi(\eta)= \frac{a\eta+b}{c\eta+d}=\eta,$$ we get
$$ a-d=c\eta-b\overline{\eta}, $$
and then
 \begin{eqnarray*}
     &&    (\overline{b}a-\overline{d}c)-  (\overline{b}d-c\overline{a})\\
     &=& \overline{b}(a-d) +c(\overline{a}-\overline{d}) \\
     &=& \overline{b}(c\eta-b\overline{\eta})+c(\overline{c}\overline{\eta}-\overline{b}\eta) \\
     &=& (|c|^2-|b|^2)\overline{\eta}\\&=&0,
  \end{eqnarray*}
Hence $$
                \frac{|d|^2}{|d|^2-|b|^2-(\overline{b}a-\overline{d}c)z}
                =\frac{|d|^2}{|d|^2-|c|^2-(\overline{b}d-c\overline{a})z}.  $$
Elementary calculation shows that
$$\varphi^*(\varphi(z))=\frac{(|a|^2-|c|^2)z+b\overline{a}-d\overline{c}}{(\overline{d}c-\overline{b}a)z+|d|^2-|b|^2},$$
$$\varphi(\varphi^*(z))=\frac{(|a|^2-|b|^2)z+b\overline{d}-a\overline{c}}{(\overline{a}c-\overline{b}d)z+|d|^2-|c|^2},$$
so we also have $\varphi^*\circ\varphi=\varphi\circ\varphi^*.$  Therefore, (\ref{14}) holds.

The proof is complete. \end{proof}

\begin{theorem}Suppose that
  $$\varphi(z)=\frac{az+b}{cz+d},\ \  ad-bc\neq 0, |b|=|c|$$
  is a linear fractional selfmap of the unit disk, which admits a boundary fixed point $\eta\in\mathbb{T}$,   and $\psi=K_{\varphi^*(0)}$
  where $\varphi^*(z)=\frac{\overline{a}z-\overline{c}}{-\overline{b}z+\overline{d}}$. Then $W_{\psi,\varphi}$ acting on $H^2(\mathbb{D})$
  is complex symmetric under the conjugation $\mathcal{J}W_{k_{p\overline{\eta}},\sigma}$, where $p$ is a nonzero point such that $bp(\overline{p}-1)=c\eta^2\overline{p}(1-p)$.
\end{theorem}
\begin{proof}
From Lemma \ref{bc}, we know $W_{\psi,\varphi}$ is normal and then is complex symmetric.

 If we denote
$$\varphi_1(z)=\overline{\eta}\varphi(\eta z)=\frac{a_1z+b_1}{z+d_1}$$ where $a_1=a\overline{\eta}/c$, $b_1=b\overline{\eta}^2/c$ and  $ d_1=d\overline{\eta}/c,$
 then $\varphi_1(1)=1$, $|b_1|=1$ and it suffices to prove that  $C_{\eta z} W_{\psi,\varphi}C_{\overline{\eta}z}$   is complex symmetric under the conjugation
 $C_{\eta z}\mathcal{J} W_{k_p,\sigma}C_{\overline{\eta}z}$. Direct calculation shows that    $$C_{\eta z} W_{\psi,\varphi}C_{\overline{\eta}z}=W_{\psi_1,\varphi_1}$$
 and   $$C_{\eta z}\mathcal{J} W_{k_{p\overline{\eta}},\sigma}C_{\overline{\eta}z}=\mathcal{J}W_{k_{p},\sigma_1},$$
where  $\psi_1(z)=\psi(\eta z)$ and $\sigma_1(z)=\overline{\eta}\sigma(\overline{\eta}z)$.

Now, we will prove  that $$ \mathcal{J}W_{k_{p_1},\sigma_1}W_{\psi_1,\varphi_1}=W_{\psi_1,\varphi_1}^*\mathcal{J}W_{k_{p_1},\sigma_1}$$
 which is equivalent to
 $$ W_{k_{p_1},\sigma_1}W_{\psi_1,\varphi_1} \mathcal{J}W_{k_{p_1},\sigma_1}=\mathcal{J}W_{\psi_1,\varphi_1}^*.$$
   Since the span of point evaluation functionals is dense in $H^2(\mathbb{D})$,  it suffices to check that
\begin{equation}\label{16}
  W_{k_{p_1},\sigma_1}W_{\psi_1,\varphi_1} \mathcal{J}W_{k_{p_1},\sigma_1}(K_w)=\mathcal{J}W_{\psi_1,\varphi_1}^*(K_w),
\end{equation}
for all $w\in\mathbb{D}$.

For the right side of (\ref{16}), we have \begin{align*}
 &\mathcal{J}W^*_{\psi_1,\varphi_1}(K_w)(z)=\mathcal{J}(\overline{\psi_1(w)}K_{\varphi_1(w)})(z)=\psi_1(w)K_{\overline{\varphi_1(w)}}(z)\\
 &\ \ \ \ \ \ \ \ \ \  \ \ \ \ \ \ \ \ \ \ \ \ \ \ =\frac{d_1}{w+d_1}\cdot\frac{1}{1-\frac{a_1w+b_1}{w+d_1}z}=\frac{d_1}{w+d_1-(a_1w+b_1)z}.
\end{align*} For the left side of (\ref{16}), we have

\begin{eqnarray*}
   && W_{k_p,\sigma} W_{\psi_1,\varphi_1} \mathcal{J}W_{k_p,\sigma}(K_w)(z) \\
   &=& W_{k_p,\sigma} W_{\psi_1,\varphi_1} \mathcal{J}\left(\frac{\sqrt{1-|p|^2}}{1-\overline{p}z}
        \cdot\frac{1}{1-\overline{w}\frac{\overline{p}}{p}\frac{p-z}{1-\overline{p}z}}\right)\\
  &=& W_{k_p,\sigma}W_{\psi_1,\varphi_1} \mathcal{J}\left(\frac{p\sqrt{1-|p|^2}}
     {(p-|p|^2\overline{w})+(\overline{p}\overline{w}-|p|^2)z}\right)\\
 &=& W_{k_p,\sigma} W_{\psi_1,\varphi_1} \left(\frac{\overline{p}\sqrt{1-|p|^2}}
   {(\overline{p}-|p|^2w)+(pw-|p|^2)z}\right)\\
  &=& W_{k_p,\sigma}\left(\frac{d_1}{z+d_1}\frac{\overline{p}\sqrt{1-|p|^2}}
    {(\overline{p}-|p|^2w)+(pw-|p|^2)\frac{a_1z+b_1}{z+d_1}}\right)\\
  &=&W_{k_p,\sigma}\left(\frac{d_1\overline{p}\sqrt{1-|p|^2}}
   {(d_1\overline{p}-b_1|p|^2-d_1|p|^2w+b_1pw)+(a_1pw-|p|^2w-a_1|p|^2+\overline{p})z}\right)\\
 &=&\frac{\sqrt{1-|p|^2}}{1-\overline{p}z}\cdot
    \left(\frac{d_1\overline{p}\sqrt{1-|p|^2}}
   {(d_1\overline{p}-b_1|p|^2-d_1|p|^2w+b_1pw)+(a_1pw-|p|^2w-a_1|p|^2+\overline{p})
   \cdot\frac{|p|^2-\overline{p}z}{p-|p|^2z}}\right)\\
 &=&\frac{d_1|p|^2(1-|p|^2)}
 {(d_1\overline{p}-b_1|p|^2-d_1|p|^2w+b_1pw)(p-|p|^2z)+(a_1pw-|p|^2w-a_1|p|^2+\overline{p})(|p|^2-\overline{p}z)}\\
 &=&\frac{d_1|p|^2(1-|p|^2)}{(Aw+B)-(Cw+D)z},
\end{eqnarray*}
where

\begin{equation*}
\begin{matrix}
 A=p^2(b_1+a_1\overline{p}-\overline{p}^2-d_1\overline{p}), \ \ \ \ & B=|p|^2(d_1+\overline{p}-a_1|p|^2-b_1p), \\
    C=|p|^2(a_1+b_1p-\overline{p}-d_1|p|^2), & D=\overline{p}^2(d_1p+1-a_1p-b_1p^2).

\end{matrix}
\end{equation*}
Using the condition $\varphi_1(1)=1$ and $b_1p(\overline{p}-1)=\overline{p}(p-1)$, one can easily verify that (\ref{16}) holds. This completes the proof.
\end{proof}

\begin{remark}
The equation $$b_1p(\overline{p}-1)+\overline{p}(1-p)=0,\ \ |b_1|=1$$
has lots of solutions in the unit disk. In fact, write $p=re^{i\theta}\neq 0$, then
$$b_1=\frac{\overline{p}(1-p)}{p(1-\overline{p})}=\frac{r-e^{-i\theta}}{r-e^{i\theta}} =e^{2i\arg(r-e^{-i\theta})}$$
which  has exactly two solutions of $\theta$ for any fixed $r\in (0,1)$.
\end{remark}

\subsection{Algebraic $W_{\psi,\varphi}$ of degree $\leq$ 2}  So far,  all complex symmetric weighted composition operators $W_{\psi,\varphi}$ have linear fractional symbols. In this part, We will show  that the weight function $\psi$ can be not linear fractional even when $\varphi$ is.
The counterexample is a weighted composition operator that is  algebraic of degree two.    The following theorem characterize when a weighted composition operator is algebraic with degree no more than two.
\begin{theorem}\label{al2}
Suppose that $\varphi\in S(\mathbb{D})$  and $\psi\in H(\mathbb{D})$ is not identically zero. Then $W_{\psi,\varphi}$
is algebraic with degree $\leq$ 2 exactly when one of the followings holds: \\
 $(1)$ $\varphi$  is a constant function;\\
 $(2)$ $\varphi$ is the identity  map and $\psi$ is a constant function;\\
 $(3)$ $\varphi(z)=\alpha_p(-\alpha_p(z))$ with $p\in\mathbb{D}$  and $$\psi\circ\alpha_p(z)= c \exp\{\sum\limits_{k=0}^{\infty}a_{2k+1}z^{2k+1}\}\in H(\mathbb{D}).$$
\end{theorem}

\begin{proof}
Suppose that $W_{\psi,\varphi}$ satisfy the equation $$AW_{\psi,\varphi}^2-BW_{\psi,\varphi}-C=0 ,$$ with $|A|^2+|B|^2+|C|^2\neq0.$

Since $\psi$ is not identically  zero,  the degree of $W_{\psi,\varphi}$ is at least one. If the degree  is one, then we have
$BW_{\psi,\varphi}+C=0,$
thus $\psi\equiv-C/B$ and $\varphi$ is identity. If the degree of $P$  is two,  we assume for simplicity that $A=1$ and
divide it into several cases to analyze.

\subsubsection*{Case I} If $B\neq0$, $C=0$, that is
$$W_{\psi,\varphi}^2-BW_{\psi,\varphi}=0.$$
Taking test functions $f\equiv1$ and $g(z)=z$, we get
\begin{align*}
\left\{
{
\begin{array}{*{18}c}
\psi\cdot \psi\circ\varphi-B\psi=0\\
\psi\cdot \psi\circ\varphi\cdot\varphi\circ\varphi-B\psi\cdot\varphi=0,\\
\end{array}
}
\right.
\end{align*}
which is equivalent to
\begin{align*}
\left\{
{
\begin{array}{*{18}c}
 \psi\circ\varphi-B=0\\
\varphi\circ\varphi-\varphi=0.\\
\end{array}
}
\right.
\end{align*}
If $\varphi$ is constant, then $\psi\equiv B$ is also constant;   otherwise, $\psi\equiv B$  and $\varphi$ is the identity map since $\varphi(\mathbb{D})$
is a nonempty open set.

\subsubsection*{Case II} If $B=0$, $C\neq0$, that is
$$W_{\psi,\varphi}^2=C.$$ Taking test functions $f\equiv1$ and $g(z)=z$, we get \begin{align*}
\left\{
{
\begin{array}{*{18}c}
\psi\cdot \psi\circ\varphi=C\\
\varphi\circ\varphi=id,\\
\end{array}
}
\right. \end{align*}
If $\varphi$ is identity, then $\psi=\pm\sqrt{C}$ is  constant.

 If  $\varphi$   is  not identity, we first consider the simple case when the fixed point of $\varphi$ is zero.
In this case,  $\varphi(z)=-z$ and $\psi(z)\cdot \psi(-z)=C\neq0$.
Setting $\psi=e^{h}$, then $$h(z)+h(-z)=\log C,$$ which implies that  the even terms  in the series  expansion of  $h-\frac{1}{2}\log C$  vanish. Consequently, $$\psi(z)=\sqrt{C}\exp(\sum_{k=0}^{\infty}a_{2k+1}z^{2k+1}).$$ For the case where $\varphi(p)=p\neq0$,    set
$\widetilde{\psi}=\psi\circ\alpha_p$ and  $\widetilde{\varphi}=\alpha_p\circ\varphi\circ\alpha_p$, then $uC_\varphi$
is similar to $\widetilde{\psi}C_{\widetilde{\varphi}}$ and hence
$P(\widetilde{\psi}C_{\widetilde{\varphi}})=0$.
So we have $\widetilde{\psi}(z)=\sqrt{C}\exp(\sum_{k=0}^{\infty}a_{2k+1}z^{2k+1}).$

\subsubsection*{Case III} If $B\neq0$, $C\neq0$, that is   $$W_{\psi,\varphi}^2-BW_{\psi,\varphi}-C=0.$$
Taking   the monomials  $\{z^n\}_{n=0}^\infty$ as test functions, we get
\begin{equation}\label{17}
  \psi\cdot \psi\circ\varphi\cdot(\varphi\circ\varphi)^n=B\psi\cdot\varphi^n+Cz^n,\;n\;\geq 0.
\end{equation}

For $n=0$, if we let $z=0$,  then  we get $$\psi(0)\cdot \psi(\varphi(0))=B\psi(0)+C,$$ thus $\psi(0)\neq0$.

For $n\geq1$, if we let   $z=0$, then  we get $$\psi(\varphi(0))[\varphi(\varphi(0))]^n=B\varphi(0)^n,$$ that is,
 $$\left[\frac{\varphi(\varphi(0))}{\varphi(0)}\right]^n=\frac{B}{\psi(\varphi(0))},\ \ n\geq 1.$$
So we  must have $\varphi(\varphi(0))=\varphi(0)$ and $B=\psi(\varphi(0))$.
 Evaluate (\ref{17}) at $\varphi(0)$,
 \begin{eqnarray*}
   \psi(\varphi(0))^2\varphi^n(\varphi(0)) &=& B\psi(\varphi(0))\varphi^n(\varphi(0))+C\varphi(0)^n \\
    &=& \psi(\varphi(0))^2\varphi^n(\varphi(0))+C\varphi(0)^n,
 \end{eqnarray*} so we get $\varphi(0)=0$.

It follows from (\ref{17}) that
\begin{equation}\label{18}
  \psi(z)\cdot \psi(\varphi(z))\left[\frac{\varphi(\varphi(z))}{z}\right]^n=B\psi(z)\left(\frac{\varphi(z)}{z}\right)^n+C,\;\; n\geq0.
\end{equation}
Since $\varphi(0)=0$, the  equation (\ref{18}), along with the Schwartz lemma, implies  that $\varphi(z)=\lambda z$ for some $|\lambda|=1$.

Now,  the equation (\ref{17}) is equivalent to
\begin{equation}\label{19}
  \psi(z)\psi(\lambda z)\lambda^{2n}=B\psi(z)\lambda^n+C\;\;n\geq 0.
\end{equation}

If $\lambda=1$, then
 $$\psi^2(z)=B\psi(z)+C.$$
 $\psi$ has to be constant  and then $W_{\psi,\varphi}$ is a multiple of the identity.

If $\lambda=-1$, then
$$\psi(z)\psi(\lambda z)-C= B\psi(z)=-B\psi(z).$$
Since $B\neq0$ and $\psi$ is not identically zero, this is impossible.

 If $\lambda$ is rational and $\lambda^N=1$ with $N$ minimal and $N\geq3$.
  Then summing (\ref{19}) with respect to $n$, we get
  $$\psi(z)\psi(\lambda z)\sum_{n=0}^{N-1}\lambda^{2n}=B\psi(z)\sum_{n=0}^{N-1}\lambda^{n}+CN.$$
  Since
  $$\lambda^N-1=(\lambda-1)(1+\lambda+\lambda^2+\cdots+\lambda^{N-1})=0$$
  and
  \begin{align*}
  & \lambda^{2N}-1=(\lambda-1)(1+\lambda+\lambda^2+\cdots+\lambda^{2N-1})\\
  &\ \ \ \ \ \ \ \ \ \  =(\lambda-1)(1+\lambda)\sum_{n=0}^{N-1}\lambda^{2n}\\
  &\ \ \ \ \ \ \ \ \ \ =(\lambda^2-1)\sum_{n=0}^{N-1}\lambda^{2n}=0,
  \end{align*}
so we have  $C=0$, which is a contradiction.

   If $\lambda$ is irrational, then
   $$\psi(z)\psi(\lambda z)w^2=B\psi(z)w+C,\ \   z\in\mathbb{D},w\in\mathbb{C}.$$
   This is obviously impossible.

The proof is complete. \end{proof}

\begin{example} Let $\varphi(z)=-z$ and $\psi_1(z)=e^{z}$, $\psi_2(z)=e^{\sin{z}}$, then $W_{\psi_1,\varphi}$ and $W_{\psi_2,\varphi}$ are
   complex symmetric on $H^2(\mathbb{D})$.
\end{example}

\vspace{8pt}

\subsection{Open Problem}
We have shown above that the weight symbol $\psi$ of a complex symmetric $W_{\psi,\varphi}$  can be not a linear fractional map. We also want to know
\begin{problem}
Is it possible for a complex symmetric weighted composition operator $W_{\psi,\varphi}$  to have its composition symbol $\varphi$  a non linear fractional map? \end{problem}
\noindent We are also  interested  in the similar problem concerning normality, see also \cite{M}.
\begin{problem}
For a normal weighted composition operator $W_{\psi,\varphi}$ with $\varphi$ a linear fractional map, must $\psi$ be also linear fractional?    Ultimately,
can we get a complete characterization of when $W_{\psi,\varphi}$ is   normal on $H^2(\mathbb{D})$?
\end{problem}

\end{document}